%% file: broyden.tex
\documentclass[final,leqno]{siamltex704}  
\usepackage[latin1]{inputenc}
\usepackage[english,swedish]{babel}
\usepackage{tikz}
\usepackage{algorithmic}
\usepackage{algorithm}
\usepackage{url}

\usetikzlibrary{external}
\usetikzlibrary{decorations.pathreplacing}
\usetikzlibrary{matrix,calc}
\usetikzlibrary{arrows}
\usetikzlibrary{decorations.markings}

\usetikzlibrary{external}
\usetikzlibrary{decorations.pathreplacing}

\usepackage{pgfplots}			
\pgfplotsset{compat=1.3}		
\newlength\figureheight                 
\newlength\figurewidth                  

\iftrue 
  \usepackage{tikz}
  \usepackage{pgfplots}
  \pgfplotsset{
    compat=newest,
    tick label style={font=\scriptsize},
    label style={font=\scriptsize},
    legend style={font=\scriptsize}
  }
  \usetikzlibrary{decorations.pathreplacing}
  \iffalse 
     \usepgfplotslibrary{external} 
  \else
     \renewcommand{\tikzsetnextfilename}[1]{}
  \fi
\else
  \usepackage{tikzexternal}
  \tikzsetexternalprefix{gfx_tmp/}
  
\fi

\usepackage{graphicx}
\usepackage{psfrag}
\usepackage[leqno]{amsmath}
\usepackage{tabularx}
\usepackage{verbatim}
\usepackage{float} 
\usepackage{amssymb}
\usepackage{latexsym}
\usepackage{pifont}
\usepackage[loose]{subfigure} 

\usepackage{color}
\newcommand{\II}{\underline{I}}

\newcommand{\RR}{\mathbb{R}}

\newcommand{\CC}{\mathbb{C}}

\newcommand{\NN}{\mathbb{N}}

\newtheorem{example}[theorem]{Example}
\newtheorem{remark}[theorem]{Remark}

\begin{document}
\title{Broyden's method for nonlinear eigenproblems}
\author{
Elias Jarlebring
%
}

\selectlanguage{english}
\maketitle
\begin{abstract}
  Broyden's method is a general method commonly used for nonlinear
  systems of equations, when very little information is available
  about the problem. We develop an approach based on Broyden's method
  for nonlinear
  eigenvalue problems. Our approach is designed for problems
  where the evaluation of a matrix vector product
  is computationally expensive, essentially
  as expensive as solving the corresponding linear system of equations.
  We show how the structure of the Jacobian matrix can be incorporated into
  the algorithm to improve convergence. The algorithm exhibits local
  superlinear convergence for simple eigenvalues, and we characterize
  the convergence. 
  We show how deflation can be integrated and combined such
  that the method can be used to compute several eigenvalues.
  A specific problem in machine tool milling, coupled with a PDE
  is used to illustrate the approach. The simulations are done in the
  julia programming language, and are provided
  as publicly  available module for reproducability.
\end{abstract}
\section{Introduction}\label{sect:intro}
We here consider the nonlinear eigenvalue problem (NEP) defined by
\begin{equation}  \label{eq:NEP}
M(\lambda)v=0
\end{equation}
where $M:\CC\rightarrow\CC^{n\times n}$ is an analytic function of $\lambda$.
This problem can equivalently be written as a system of
nonlinear equations
\begin{equation}\label{eq:F0}
  F\left(\begin{bmatrix}
    v\\
    \lambda
  \end{bmatrix}\right)=0
\end{equation}
where
\begin{equation}\label{eq:Fdef}
  F\left(\begin{bmatrix}
    v\\
    \lambda
  \end{bmatrix}\right):=\begin{bmatrix}M(\lambda)v\\c^Hv-1\end{bmatrix}
\end{equation}
under the assumption that $c$ is not orthogonal to the eigenvector.
The normalization condition $c^Hx=1$
is selected such that $F$ is analytic and
therefore complex differentiable, which would not be the
case if we  were to select $\|v\|_2^2=v^Hv=1$ instead.

This class of NEPs has been studied for decades, as
can be seen in summary references \cite{Ruhe:1973:NLEVP,Mehrmann:2004:NLEVP,Voss:2013:NEPCHAPTER}
and the benchmark collection \cite{Betcke:2010:NLEVPCOLL}.
Several standard approaches for NEPs of the type \eqref{eq:NEP}
are based on Newton's method. The Newton approach for NEPs
was  proposed already in 1950 \cite{Unger:1950:NICHTLINEARE},
and later developed further in \cite{Peters:1979:INVERSE,Ruhe:1973:NLEVP}.
The residual inverse iteration \cite{Neumaier:1985:RESINV} is an implicit
Newton method \cite{Jarlebring:2017:QUASINEWTON} and forms the basis
of the nonlinear Arnoldi method \cite{Voss:2004:ARNOLDI}.
More recently, block variants of Newton's method has been developed \cite{Kressner:2009:BLOCKNEWTON}.
There is a summary of many methods \cite{Guettel:2017:NEP} of which 
many are Newton methods or can be interpreted as flavors of Newton's method.
The QR-approach for
banded matrices in 
\cite{Garrett:2016:NQA} is based
on Kublanovskaya's approach \cite{Kublanovskaya:1970:APPROACH}
which is also a Newton method
applied to the $(n,n)$-element of the  R-matrix in the QR-factorization of $M(\lambda)$. Two-sided Newton approaches and Jacobi-Davidson approaches have been studied in \cite{Schreiber:2008:PHD}.
Considerable convergence theory and specialization of the Newton type approaches can be found in the literature, e.g., convergence theory 
\cite{Szyld:2013:LOCAL, Szyld:2015:LOCALI,Szyld:2015:LOCALII}
as well as inexact solves and preconditioning \cite{Szyld:2011:EFFICIENT}.

These Newton-approaches depend on explicit access to the matrix $M(\lambda)$,
in ways which are not available.
Most methods depend on direct access of $M(\lambda)$ and/or that the
NEP can be expressed in an affine form 
\begin{equation}  \label{eq:affine}
   M(\lambda)=M_1f_1(\lambda)+\cdots+M_mf_m(\lambda)
\end{equation}
where $f_1,\ldots,f_m$ are analytic functions and $m\ll n$.
The availability of an affine form typically means that when $m$ is small,
the projected problem $V^TM(\lambda)Wz=0$ can normally be solved in
a computationally cheap way.
The matrix $M(\lambda)$ and an affine form are not always available
in  applications. We illustrate this further in 
Section~\ref{sec:TPDDE} with a problem stemming from the analysis of time-periodic
delay-differential equations.

The approach presented here is based on Broyden's method for nonlinear
systems of equations; see \cite{Broyden:1965:BROYDEN} and more recent
summaries in \cite{Deuflard:2004:NEWTON,Al-Baali:2014:BROYDEN}. Broyden's method is also based on Newton's method, but the Jacobian approximation is updated
(typically with a rank-one matrix)
in order to avoid the computation of the Jacobian matrix. An attractive feature 
of Broyden's method is that only one function evaluation per iteration is required. In the context of NEPs this implies that we do not need an affine form and nor a direct
accurate access to the Jacobian matrix. 

In common for many structured iterative methods, application of a general purpose approach
to a specific problem leads to structures which can be exploited in
the algorithm. We derive in Section~\ref{sec:structbroyden}
a structure of Broyden method iterates when applied to \eqref{eq:F0},
which allows us to improve the approach.
We show how this can be integrated with a
deflation technique (in Section~\ref{sec:deflation}).
In this context we also show how restarting can be carried out in
a natural way.
A local convergence is also characterized (in Section~\ref{sec:convergence}).
We show how the convergence is related to Jordan structure
in the sense of \cite{Gohberg:1982:MATRIXPOLYNOMIALS,Hryniv:1999:PERT}.
More precisely, we show how the convergence is given by the
Jordan chains defined as  the existance of solutions to the equation 
\[
   \sum_{i=0}^a\frac{M^{(i)}(\lambda)}{i!}v_{a-i}=0
   \]
where $v_0$ is a singular vector of $M(\lambda)$.

We present numerical results of simulations for several problems in
Section~\ref{sec:simulations} and Section~\ref{sec:TPDDE} in
order to illustrate the properties of the method and its
competitiveness for the time-perioidic time-delay system.

\section{Background and basic algorithm}
We briefly summarize the
specific version of Broyden's method which
will be the basis of our algorithm on. 
%
We use a damped version of Broyden's method, as described
e.g., in  \cite[Section 7]{Al-Baali:2014:BROYDEN}.
The derivation follows from the Newton-like update
equation 
\begin{equation}  \label{eq:DeltaxkdefJ}   
J_k\Delta x_k=-F(x_k)
\end{equation}
where the next approximation is computed with a damped update equation
\begin{equation} \label{eq:xk_update}
  x_{k+1}=x_k+\gamma_k \Delta x_k  
\end{equation}
The choice of the damping parameter $\gamma_k$ will be tuned to our
setting, essentially to avoid taking too large steps (as
we shall further describe in Remark~\ref{rem:damping}).
  The next matrix $J_{k+1}$ will satisfy (what is commonly called)
  the secant condition
\begin{equation}  \label{eq:secanteq}
J_{k+1}(x_{k+1}-x_k)=F(x_{k+1})-F(x_k)
  \end{equation}
where $J_{k+1}$ is a rank-one modification of $J_{k}$.
We will focus on updates of the form,
\begin{equation}  \label{eq:Jkupdate_structure}
J_{k+1}=J_k+\frac{1}{\|\Delta x_k\|^2}z_{k+1}\Delta x^H.
\end{equation}
By combining \eqref{eq:DeltaxkdefJ},
\eqref{eq:secanteq} and \eqref{eq:Jkupdate_structure},
it is clear that $z_{k+1}$ can be directly computed from 
\begin{equation}  \label{eq:zdef}
  z_{k+1}=\frac{1}{\gamma_k}(F(x_{k+1})-(1-\gamma_k)F(x_k)).
\end{equation}
In the literature on Broyden's method (without damping), e.g.,
the original work \cite{Broyden:1965:BROYDEN}, 
the relation \eqref{eq:Jkupdate_structure}
with choice \eqref{eq:zdef} is typically viewed as the
minimization of the update matrix $J_{k+1}-J_k$  with respect
to the Frobenius norm and maintaining the secant
condition \eqref{eq:secanteq}. 

The equations \eqref{eq:DeltaxkdefJ}, \eqref{eq:xk_update},
\eqref{eq:zdef} and \eqref{eq:Jkupdate_structure}
form an explicit algorithm where the state consists
of a vector $x_k$ and a matrix $J_k$, taking the role
of a Jacobian matrix. This algorithm is called Broyden's good method.
(Our algorithm can be modified to carry out bad Broyden's method.
We focus on the good Broyden method, for simiplicity.)
An unfavorable aspect from a computational perspective
is that the linear system in \eqref{eq:Deltaxkdef} needs to
solved in every step.  There are several ways to avoid this.
Instead storing with the inverse of $J_k$ we can store its inverse
\[
  H_k=J_k^{-1}.
\]
and state the algorithm in terms of $H_k$ instead of $J_k$.
We see immediately that \eqref{eq:DeltaxkdefJ} becomes
\begin{equation}  \label{eq:Deltaxkdef}
\Delta x_k=-H_kF(x_k)  
\end{equation}
Similarly, the update equation \eqref{eq:Jkupdate_structure}
can be reformulated in terms of $H_k$. More precisely,
by applying the Sherman-Morrison-Woodbury formula
\cite[Section~2.1.4]{Golub:2007:MATRIX}, we obtain
\begin{subequations}
\begin{eqnarray}
  H_{k+1}&=&J_{k+1}^{-1}=\left(J_k+\frac{1}{\|\Delta x_k\|^2}z_{k+1}\Delta x_k^H\right)^{-1}\\
  &=&J_k^{-1}- \frac{J_k^{-1}z_{k+1}\Delta x_k^HJ_k^{-1}}{\|\Delta x_k\|^2+\Delta x_k^HJ_k^{-1}z_{k+1}}\\
  &=&H_k- \frac{H_kz_{k+1}\Delta x_k^HH_k}{\|\Delta x_k\|^2+\Delta x_k^HH_kz_{k+1}}.
\end{eqnarray}
\end{subequations}
By using \eqref{eq:zdef} and \eqref{eq:secanteq}
we see that
$H_kz_{k+1}=\frac{1}{\gamma} (H_kF(x_{k+1})+(1-\gamma)\Delta x_k)$
and the following equivalent alternative relation for $H_{k+1}$
\begin{equation} 
  H_{k+1}=H_k-
  \frac{(H_kF(x_{k+1})+(1-\gamma)\Delta x_k)\Delta x_k^HH_k}%
  {\Delta x_k^H(H_kF(x_{k+1})+\Delta x_k)}
\end{equation}
\begin{example}\label{exmp:roundoff}
  In order to illustrate the differences between the two versions of
  Broyden's method in terms of round-off error we carry out
  simulations on a small example (for reproducability). We consider
  the quadratic eigenvalue problem
  \[
 M(\lambda)=A_0+A_1\lambda+A_2\lambda^2
  \]
  where $A_0$, $A_1$ and $A_2$ were randomly generated. 
  We carried out the simulation for both $H$-version
  and $J$-version in single precision,
  as well as a simulation in
  sufficiently high precision such that the result
  iteration can be treated as exact. The residual
  norm history is given in Figure~\ref{fig:roundoff}.
  We see that the $H$-version follows
  the exact error history (computed with high
  precision arithmetic) worse than the
  $J$-version.   The algorithm presented in the next section
  follows the trajectory even better ($T$-variant).
  Although the differences between the methods
  are small in this example, it illustrates what can
  be seen in longer simulations (in Section~\ref{sec:simulations}). 
  \end{example}

\begin{figure}[h]
  \begin{center}
    \subfigure[Iteration history]{\input{roundoff_conv1.tex}}%
    \subfigure[Zoomed iteration history]{\input{roundoff_conv.tex}}
    \caption{Round-off error illustration.   
      \label{fig:roundoff}
    }
  \end{center}
\end{figure}
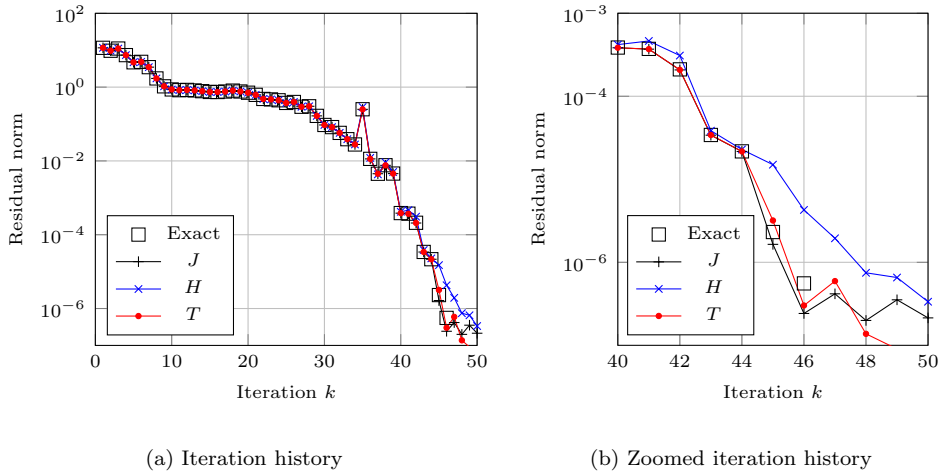

\section{Structure exploiting Broyden method}\label{sec:structbroyden}
\subsection{Structure of the iterates}
We now consider nonlinear systems of equations
with a particular structure:
\begin{equation}  \label{eq:structF}
  F\left(\lambda,
  \begin{bmatrix}
   v\\u
  \end{bmatrix}\right)=
  \begin{bmatrix}
    M(\lambda)&U(\lambda)\\
    C^H& 0
  \end{bmatrix}\begin{bmatrix}
    v\\
    u
  \end{bmatrix}-
  b\in\CC^{n+p+1}
\end{equation}
where $C^H\in\CC^{(p+1)\times n}$, $U(\lambda)\in\CC^{n\times p}$,
$v\in\CC^n$ and $u\in\CC^p$. We will also consistently partition $b$
as $b^T=\begin{bmatrix} b_1^T& b_2^T\end{bmatrix}$.

This structure includes the nonlinear equation formulation
in \eqref{eq:Fdef} as the special case $p=0$ and $b=e_{n+1}$.
We take this more general approach in order to
incorporate deflation in a natural way, as
we will describe in Section~\ref{sec:deflation}.
The Jacobian of this problem can be derived explicitly, 
\begin{equation}  \label{eq:Jstruct}  
  J\left(\lambda,
  \begin{bmatrix}
   v\\u
  \end{bmatrix}\right)
  =\begin{bmatrix}
  M(\lambda)&U(\lambda)&M'(\lambda)v+U'(\lambda)u\\
  C^H&0&0\\
  \end{bmatrix}\in\CC^{(n+p+1)\times (n+p+1)}
\end{equation}
We first note that the structure of the Jacobian
and the iterates are preserved in Broyden's method, when we
denote
\begin{equation}\label{eq:xblocks}
x_k=\begin{bmatrix}
v_k\\
u_k\\
\lambda_k
\end{bmatrix}.
\end{equation}
More precisely, if we initialize the Jacobian in Broyden's
method with the  structure, and label the blocks as
  \begin{equation} \label{eq:J0struct}
    J_1
=
\begin{bmatrix}
  M_1& W_1 \\
  C^H& 0
  \end{bmatrix}
      \in\CC^{(n+p+1)\times (n+p+1)}.
  \end{equation}
  where $W_1=\begin{bmatrix}U_1&f_1\end{bmatrix}$,
  then this structure is preserved in the sense of the
  following theorem.
\begin{theorem}[Structured iterates of Broyden's method]\label{thm:structure}
  Let $(v_1,u_1,\lambda_1)$ be such that, 
  \begin{equation}  \label{eq:CTc0}
  C^Hv_1=b_2
  \end{equation}
  and  $J_1$ be set to \eqref{eq:J0struct}.
  Suppose Broyden's method initiated with   $(v_1,u_1,\lambda_1)$
  and $J_1$ applied to \eqref{eq:structF} does not break down, and let $(v_k,u_k,\lambda_k)$ and $J_k$, $k=2,\ldots$ be
  the iterates.
  Then, the structures \eqref{eq:CTc0} and \eqref{eq:J0struct} are preserved for all $k$, i.e.,
  for  $k=2,\ldots,$ we have
  \begin{equation}  \label{eq:Jkstruct}
      J_k=
\begin{bmatrix}
  M_k& W_k \\
  C^H& 0
  \end{bmatrix}  
  \end{equation}
and
\begin{equation}  \label{eq:CTck}
  C^Hv_k=b_2.  
\end{equation}
  
\end{theorem}
\begin{proof}
  The proof is by induction.
  We suppose \eqref{eq:Jkstruct} and \eqref{eq:CTck} for a specifik $k$
  and prove these two equations for $k+1$.
  It is clear  from \eqref{eq:DeltaxkdefJ} that  $\Delta x_k$ satisfies
  \[
  \begin{bmatrix}
    M_k&W_k\\
    C^H& 0 
  \end{bmatrix}
  \Delta x_k
  =\begin{bmatrix}
    M(\lambda_k)v_k+U(\lambda_k)u_k-b_1\\
  C^Hv_k-b_2
  \end{bmatrix}=\begin{bmatrix}
    M(\lambda_k)v_k+U(\lambda_k)u_k-b_1\\
0
  \end{bmatrix}
  \] 
  such that $C^H\Delta v_k=0$. 
  Since $v_{k+1}=v_k+\gamma \Delta v_k$,  we have
  \[   
  C^Hv_{k+1}=C^H(v_k+\gamma \Delta v_k)=C^Hv_k=b_2.
  \]
  which shows \eqref{eq:CTck} for $k+1$. Therefore, the vector $z_{k+1}$ has
  the structure
  \begin{multline*}
  z_{k+1}=\frac{1}{\gamma}\left(
  \begin{bmatrix}
    M(\lambda_{k+1})v_{k+1}+U(\lambda_{k+1})-b_1\\
    C^Hv_{k+1}-b_2
  \end{bmatrix}-
(1-\gamma_k)
  \begin{bmatrix}
    M(\lambda_{k})v_{k}+U(\lambda_{k})-b_1\\
    C^Hv_{k}-b_2
  \end{bmatrix}
  \right)=\\
\frac{1}{\gamma}
  \begin{bmatrix}
    M(\lambda_{k+1})v_{k+1}+U(\lambda_{k+1})-b_1
-(1-\gamma_k)(M(\lambda_{k})v_{k}+U(\lambda_{k})-b_1)
    \\
 0
  \end{bmatrix}
  \end{multline*}
  The matrix $J_k$ is updated according to \eqref{eq:Jkupdate_structure}.
  The last block row of the update in \eqref{eq:Jkupdate_structure} is
  zero, since the last block row of is $z_{k+1}=0$. Therefore,
  we can define some $J_{k+1}$ and $W_{k+1}$ such that  \eqref{eq:Jkstruct}
  is satisfied for $k+1$.
      \end{proof}
%


\subsection{Structured Broyden}

With the objective to improve Broyden's method for
nonlinear systems of equations of the form \eqref{eq:structF}, we now show how
the structure proven in Theorem~\ref{thm:structure}
can be implicitly preserved.
The $J$-version is straightforward to
modify to incorporate the structure,
by consideration of the blocks of \eqref{eq:DeltaxkdefJ}
in $J_k\Delta x=-F(x_k)$
as follows. We multiply the first block row of 
equation \eqref{eq:DeltaxkdefJ} from the left with
$C^HM_k^{-1}$, i.e.,
\begin{equation}  \label{eq:CMfirst}  
  C^HM_k^{-1}\left(M_k\Delta v_k+W_k\begin{bmatrix}
  \Delta u_k\\
  \Delta \lambda_k
\end{bmatrix}\right)=-  C^HM_k^{-1}r_k,
\end{equation}
  where the residual $r_k$ is defined as 
\begin{equation}  \label{eq:rkdef}
r_k=M(\lambda_k)v_k+U(\lambda_k)u_k-b_1.
\end{equation}
By using that $C^H\Delta v_k=C^H(v_{k+1}-v_k)/\gamma_k=(b_1-b_1)/\gamma_k=0$
due Theorem~\ref{thm:structure},
we conclude from \eqref{eq:CMfirst} that
the following linear system
for $\Delta u_k$ and $\Delta \lambda_k$
is satisfied
\begin{equation}\label{eq:Deltalambdau_JT}
-C^HM_k^{-1}r_k=(C^HM_k^{-1} W_k)\begin{bmatrix}
  \Delta u_k\\
  \Delta \lambda_k
\end{bmatrix}.
\end{equation}
Subsequently, $\Delta v_k$ is found from the first block
row of \eqref{eq:DeltaxkdefJ}, i.e.,
\begin{equation}\label{eq:Delta_vk_JT}
\Delta v_k=-M_k^{-1}\left(W_k
\begin{bmatrix}
  \Delta u_k\\
  \Delta \lambda_k
\end{bmatrix}+r_k\right)
\end{equation}
Hence, the solution of the
linear system in \eqref{eq:DeltaxkdefJ} can be
replaced by first solving \eqref{eq:Deltalambdau_JT}
and then computing \eqref{eq:Delta_vk_JT}. This procedure can be implemented
with $p+2$ linear solves.

At first sight, nothing is gained since we need even more linear solves
than the $J$-version. However, similar to the $H$-version, we can now formulate the algorithm
by representing an inverse. More precisely, instead of storing $M_k$ we store,
\[
  T_k=M_k^{-1}.
\]
The reasoning with exploitation of the Jacobian in the $J$-version
can be translated as follows. 
Equation \eqref{eq:Deltalambdau_JT} can be replaced
by computing 
\begin{equation} \label{eq:Zkdef}
Z_k=T_kW_k  
\end{equation}
which allows us to compute the corresponding linear system in $p+1$ unknowns: 
  \begin{equation}  \label{eq:Deltalinsys}    
\begin{bmatrix}
  \Delta u_k\\
  \Delta \lambda_k
\end{bmatrix}=
-(C^HZ_k)^{-1}(C^HT_kr_k),
  \end{equation}
from which we can form
\begin{equation}  \label{eq:Deltavk}
\Delta v_k=-Z_k
\begin{bmatrix}
  \Delta u_k\\
  \Delta \lambda_k
\end{bmatrix}-T_kr_k.
\end{equation}
For notational convenience we now set
$\II^H:=\begin{bmatrix}
I& 0 \end{bmatrix}\in\RR^{n\times (n+p+1)}$.
After updating the iterates
\begin{subequations}\label{eq:update}
\begin{eqnarray}
  v_{k+1}&=&v_k+\gamma\Delta v_k   \\
  u_{k+1}&=&u_k+\gamma\Delta u_k\\
  \lambda_{k+1}&=&\lambda_k+\gamma\Delta \lambda_k
\end{eqnarray}
\end{subequations}
we compute a new residual  corresponding to $r_{k+1}$
using \eqref{eq:rkdef}  and define $\tilde{z}_{k+1}$
as
\begin{equation}  \label{eq:tildezk}
\tilde{z}_{k+1}=\II^Tz_{k+1}=\frac{1}{\gamma_k}(r_{k+1}-(1-\gamma_k)r_k).  
\end{equation}

By again applying the Sherman-Morrison-Woodbury formula, we see that we can directly
update $T_k$
\begin{subequations}
\begin{eqnarray*}
T_{k+1}&=&M_{k+1}^{-1}=\left(M_k+\frac{1}{\|\Delta x_k\|^2} \II^Tz_{k+1}\Delta x_k^H\II\right)^{-1}\\
&=& T_k-
\frac{1}{\|\Delta x_k\|^2+\Delta x_k^H\II T_k \II^Tz_{k+1}}T_k\II^Tz_{k+1}\Delta x_k^H\II T_k
\end{eqnarray*}
\end{subequations}
which can be further simplified to not contain $x$-dependence,
\begin{equation}  \label{eq:Tkplus1}
  T_{k+1}=T_k+
  T_k \tilde{z}_{k+1}a_{k+1}^H
\end{equation}
where $a_{k+1}^H:=- \Delta v_k^HT_k/(\|\Delta v_k\|^2+\|\Delta u_k\|^2+|\Delta\lambda_k|^2+\Delta v_k^HT_k \tilde{z}_{k+1})$. 
We can subsequently update $W_k$ with
  \begin{equation}  \label{eq:W_update}
      W_{k+1}=W_k+\tilde{z}_{k+1}b_{k+1}^H
  \end{equation}
  where $b_{k+1}^H=\begin{bmatrix}
    \Delta u_k^H & (\Delta \lambda_k)^H
  \end{bmatrix}/(\|\Delta v\|^2+\|\Delta u\|^2+|\Delta\lambda|^2)$.
  Finally, as a consequence of the fact that
  $T_k$ and $W_k$ are updated with rank-one matrices,
  we can also compute $Z_{k+1}$ by a rank one update of  $Z_k$ rather than using the
  definition \eqref{eq:Zkdef}. By combining
  \eqref{eq:Tkplus1} and \eqref{eq:W_update} we find that
  \begin{equation} \label{eq:Z_update}
    Z_{k+1}=T_{k+1}W_{k+1}=Z_k+T_k\tilde{z}_k(a_{k+1}^HW_k+(1+a_{k+1}^H\tilde{z}_{k+1})b_{k+1}^H).
\end{equation}

  We now note that the above equations form an algorithm,
  which does not contain explicitly $J_k$, nor $x_k$,
  and implicitly preserves the preserves the Jacobian
  structure in Theorem~\ref{thm:structure}. The algorithm
  is summarized in Algorithm~\ref{alg:structbroyden}.
  For implementation details, such as how to update $T_k$,
  $W_k$ and $Z_k$ by using only two vector operations, we
  refer to the publicly available software, further described in Section~\ref{sec:simulations}.
  As a consequence of the derivation, we have the following equivalence.
  
\begin{theorem}[Equivalence Broyden methods]
  The $J$-version of Broyden's method applied to \eqref{eq:structF},
  i.e., the iteration defined by \eqref{eq:DeltaxkdefJ}, \eqref{eq:xk_update}
  and \eqref{eq:secanteq}
  is equivalent to the structured Broyden's method, i.e.,
  the iteration defined by \eqref{eq:Deltalinsys},
  \eqref{eq:Deltavk}, \eqref{eq:update}, \eqref{eq:Tkplus1}, \eqref{eq:W_update} and \eqref{eq:Z_update}.
  Moreover, the states of the algorithms are related by \eqref{eq:xblocks} and 
  \begin{equation}  \label{eq:JkstructT}
      J_k=
  \begin{bmatrix}
    T_k^{-1}&W_k\\
    C^H&0
  \end{bmatrix}.
  \end{equation}
  \end{theorem}

\renewcommand{\algorithmicrequire}{\textbf{Input:}}
\renewcommand{\algorithmicensure}{\textbf{Output:}}
\begin{algorithm}
\begin{algorithmic}[1]
  \REQUIRE Starting values:\\
  Vectors: $v_1\in\CC^n$, $u_1\in\CC^p$, $\lambda_1\in\CC$ approximating solution to \eqref{eq:structF}\\
  Matrices: $C^H\in\CC^{(p+1)\times n}$, $T_1\in\CC^{n\times n}$ and $W_1\in\CC^{n\times (p+1)}$ approximating \eqref{eq:Jstruct} \\
  Input must satisfy $C^Hv_1=b_2$.
\ENSURE $v_m$, $u_m$, $\lambda_m$, $T_m$, $W_m$
\STATE Compute $r_1$ according to \eqref{eq:rkdef}
\STATE Compute $Z_1$ according to \eqref{eq:Zkdef}
\WHILE {$k=1,2,\ldots$ until convergence }
\STATE Compute $\Delta u_k$, $\Delta \lambda_k$ by solving the linear system  \eqref{eq:Deltalinsys} in $p+1$ variables
\STATE Compute $\Delta v_k$ with \eqref{eq:Deltavk}
\STATE Select the damping parameter $\gamma_k$, e.g., as in Remark~\ref{rem:damping}
\STATE Update the iterates by computing $v_{k+1}$, $u_{k+1}$ and $\lambda_{k+1}$ with \eqref{eq:update}
\STATE Compute $r_{k+1}$ according to \eqref{eq:rkdef}
\STATE Compute $\tilde{z}_{k+1}$ using $r_k$ and $r_{k+1}$ and \eqref{eq:tildezk}
\STATE Compute $Z_{k+1}$ with \eqref{eq:Z_update}
\STATE Compute $W_{k+1}$ with \eqref{eq:W_update}
\STATE Compute $T_{k+1}$ with \eqref{eq:Tkplus1}
\ENDWHILE
\end{algorithmic}
\caption{Structured Broyden's method\label{alg:structbroyden}}
\end{algorithm}
\bigskip\bigskip

\begin{remark}[Selection of damping]\label{rem:damping}
  The damping parameter is used to prevent the algorithm
  from taking too big steps in a pre-asymptotic phase, which can otherwise
  lead divergence or convergence to an (undesired) solution far away.
  In practice, we observed that the $\lambda$-approximation
  in the beginning of the iteration often generated new approximations far
  away from the true solution,
  Therefore, we capped the step by selection
  \begin{equation}  \label{eq:gammak_selection}
  \gamma_k=\min(1,t/|\Delta x_k|)    
  \end{equation}
  where $\tau$ is a threshold parameter.
  
  This implies $\|x_{k+1}-x_k\|<t$ and in particular that $|\lambda_{k+1}-\lambda_k|\le t$.
  This choice was determined based on numerical simulations.
  Another option  would be the Armijo-steplength, as
  used, e.g., in the context of Newton's method for NEPs in \cite{Kressner:2009:BLOCKNEWTON}.
  In contrast to \eqref{eq:gammak_selection},
  the standard implementation of Armijo step-steplength involves
  function evaluations, and is not competitive in our situation. 
  We note that there is very little general conclusive theoretical
  analysis concerning how the damping paramater is best chosen in
  a Broyden setting, as e.g., pointed
  out in \cite{Al-Baali:2014:BROYDEN}.
\end{remark}

%
%
%
%
%
%
%
%
%
%
%
%
\section{Deflation}\label{sec:deflation}

\subsection{A deflated NEP}
Structured Broyden's method can be directly applied to \eqref{eq:Fdef}
to compute an eigenpair of \eqref{eq:NEP}, as
was illustrated in Example~\ref{exmp:roundoff}.
  In order to provide the possibility
to compute several eigenvalues in a robust way, we
here develop a deflation technique, which can be integrated
with the structured Broyden's method. Our reasoning
is inspired by the work on invariant
pairs for NEPs in \cite{Kressner:2009:BLOCKNEWTON} and
deflation \cite{Effenberger:2013:PHD}.
These works, in turn, are inspired by ideas for quadratic eigenvalue
problems
\cite{Meerbergen:2001:LOCKING,Meerbergen:2008:QUADARNOLDI,Beyn:2009:CONTINUATION}. 

The essential conclusion of our reasoning provided below is that we can 
define an augmented NEP as
\begin{equation}\label{eq:MNEP}
G(\lambda)=
\begin{bmatrix}
  M(\lambda) & U(\lambda)\\
  X^H & 0 
  \end{bmatrix}
\end{equation}
whose eigenvalues are essentially  the same as the original NEP
except for some eigenvalues which are removed.
(We postpone the definition of $X$ and the function $U(\lambda)$ until
after the discussion of invariant pairs below.)
Note that if we add an orthogonalization constraint
to $G(\lambda)w=0$ as in \eqref{eq:F0} with a particular vector $c^H=[c_1^H\;\;0]$,
we obtain a nonlinear system of with the structure of the previous
section, i.e., \eqref{eq:structF}. Our
construction is based on applying Algorithm~\ref{alg:structbroyden}
to this problem.

For the derivation of this approach we need the
concepts of invariant pairs, orthogonalization conditions
and augmented invariant pairs,
which we briefly summarize. See \cite{Kressner:2009:BLOCKNEWTON},
\cite{Effenberger:2013:PHD} and \cite{Effenberger:2013:ROBUST}
for a detailed characterization.
Without loss of generality, let $M$ in \eqref{eq:NEP} be decomposed
as a sum of products of matrices and functions as in
\eqref{eq:affine}.
This decomposition always exists, although in computation it
does not always lead to efficient algorithms if $m$ is large.
We will only use this decomposition for theoretical purposes
and not in the final algorithm.
An invariant pair of \eqref{eq:NEP} is defined as a
pair $(X,S)\in\CC^{n\times p}\times \CC^{p\times p}$ which satisfies
\[
   0=M_1Xf_1(S)+\cdots+M_mXf_m(S),
\]
where $f_i(S)\in\CC^{m\times m}$ are matrix functions of $f_i$, $i=1,\ldots,m$.
By computing a Schur decomposition of $S$, it is possible
to show that the eigenvalues of $S$ are eigenvalues of \eqref{eq:NEP}.
For standard eigenvalue problems, we usually require that the columns of $X$
(which form a basis of an invariant subspace)
are linearly independent. This is done
in order to prevent the same eigenspace to appear several times
in the invariant pair.
In practice (still linear eigenvalue problems) this is usually achieved
by imposing that
the columns of $X$  are orthonormal.
The concept of minimality formalizes this reasoning.
The minimality concept is slightly different in the nonlinear
case,  
due to the fact that several eigenvalues can have
the same eigenvector  (or correspondingly for invariant subspaces).
The generalization is not expressed terms of the  column span $X$, but instead of the column span of
\begin{equation}  \label{eq:longorth}
  \begin{bmatrix}
  X\\
  \vdots\\
  XS^{\ell-1}
\end{bmatrix}.
\end{equation}
If there exists $\ell\in\NN$ such that \eqref{eq:longorth}
has full column rank, then the pair is called minimal, and the smallest $\ell$
such that \eqref{eq:longorth} has full column rank,
is called the minimality index of the pair $(X,S)$.
As pointed out in \cite{Effenberger:2013:ROBUST}, for minimal invariant
pairs $\ell=1$ is generic.
  
The concept of invariant pairs was used in a natural way to
construct a deflation technique for (simplified) Newton method and a Jacobi-Davidson
method in \cite{Effenberger:2013:ROBUST} and \cite{Effenberger:2013:PHD}.
The main idea is to compute invariant pairs one column at a time.
Given an invariant pair $(X,S)$, vectors $v$, $u$ and $\lambda$
are computed such that the extended pair
  \begin{equation}  \label{eq:extendedpair}
   (\hat{X},\hat{S})=
  \left(\begin{bmatrix}
    X&v
  \end{bmatrix},
\begin{bmatrix}
  S & u\\
  & \lambda
  \end{bmatrix}
  \right)
  \end{equation}
is also an invariant pair.
In \cite[Lemma~6.1.3]{Effenberger:2013:PHD} the
minimality is guaranteed by imposing orthogonality
  to the columns of \eqref{eq:longorth}, 
  \begin{equation}  \label{eq:orthogonality}
  \begin{bmatrix}
  X\\
  \vdots\\
  XS^{\ell-1}
\end{bmatrix}^H    \begin{bmatrix}
  \hat{X}\\
  \vdots\\
  \hat{X}\hat{S}^{\ell-1}
\end{bmatrix}e_{p+1}=0
  \end{equation}
  In this way, we avoid
  \emph{reconvergence}, i.e.,
  if an eigenvalue is contained in $S$, the algorithm
  will not find this eigenvalue again, unless it has multiplicity greater than one.

 The condition that the  extended pair \eqref{eq:extendedpair} is invariant,
 is equivalent to a more explicit condition, shown in the following lemma.
\begin{lemma}[Lemma 6.1.1 of \cite{Effenberger:2013:PHD}]\label{lem:effenberger}Let $(X,S)$ be an
  invariant pair of the nonlinear eigenvalue problem \eqref{eq:NEP}.
  Then, the extended pair \eqref{eq:extendedpair}
  is an invariant pair if and only if
  \begin{equation}  \label{eq:invpairTU}
 M(\lambda)v+U(\lambda)u=0    
  \end{equation}
 where
 \begin{equation}  \label{eq:Udef}
 U(\lambda)=\frac{1}{2\pi i}\oint_\Gamma M(\xi)(\xi I-\Lambda)^{-1}(\xi-\lambda)^{-1}\,d\xi.
 \end{equation}
\end{lemma}
If $\lambda\not\in\lambda(S)$ we have additionally (as formalized in \cite[Lemma~6.2.2]{Effenberger:2013:PHD}) 
\begin{equation}  \label{eq:Udef2}
U(\lambda)=M(\lambda)X(\lambda I-S )^{-1}.
\end{equation}
In this work we will in practice extensively use \eqref{eq:Udef2}
rather than the slightly more general definition \eqref{eq:Udef}. 

We now focus on the case $\ell=1$; see Remark~\ref{rem:ind2} for discussion of general case.
By combining equation \eqref{eq:invpairTU} and \eqref{eq:orthogonality}
we reach the nonlinear eigenvalue problem corresponding to \eqref{eq:MNEP}
  where  $\begin{bmatrix} v^T&u^T  \end{bmatrix}^T$ is an eigenvector of
  the NEP $G(\lambda)$, given by \eqref{eq:MNEP}.

  This reasoning is formalized in the following theorem,
  which can be interpreted as a complement to
  \cite[Theorem~3.6]{Effenberger:2013:ROBUST}
  where we also stress that imposing orthogonality is not restricting
  the set of minimal invariant pairs. 
  We state the theorem in terms of similarity transformations.
  The pair $(X,S)$ is a minimal invariant pair, if and only if $(XZ,Z^{-1}SZ)$ is a
  minimal invariant where $Z\in\CC^{p\times p}$ is invertible
  \cite[Lemma~3.2.3]{Effenberger:2013:PHD}. We say that $(X,S)$ and $(XZ,Z^{-1}SZ)$ are
  \emph{equivalent by similarity transformation}.
  

  \begin{theorem}[Index one extensions]\label{thm:equiv}
    Suppose $(X,S)$ is a minimal invariant pair with index one.
    Then, all minimial invariant pairs with index one
    of the form \eqref{eq:extendedpair} are
    equivalent by similarity transformation
    to the minimal  invariant pairs with index one of the form \eqref{eq:extendedpair}
    where $v,u,\lambda$ are solutions to
\begin{equation}  \label{eq:extNEP}
  \begin{bmatrix}
  M(\lambda) & U(\lambda)\\
  X^H & 0 
\end{bmatrix}
\begin{bmatrix}
  v\\u\end{bmatrix}=0,
\end{equation}
and $\|v\|+\|u\|\neq 0$.
\end{theorem}
  \begin{proof}
    Let $Z=R^{-1}P$ be the similarity transformation, defined
    by the QR-factorization of $X=QR$ (where $Q\in\CC^{n\times p}$
    and $R\in\CC^{p\times p}$ is invertible
    since the columns of $X$ are linearly independent) and
    the Schur factorization $RSR^{-1}=PRP^H$.
    From this transformation we see that
    $(X,S)$ is equivalent by similarity transformation
    to $(XZ,Z^{-1}SZ)$ where $XZ$
    is orthogonal and $Z^{-1}SZ$ upper triangular. By a change
    of variable, the condition
    \eqref{eq:extNEP} is unmodified by the transformation. Hence,
    without loss of generality we can assume that $X$ is
    orthogonal and $S$ upper triangular.
   
    Suppose $(\hat{X},\hat{S})$ is an augmented
    minimal invariant pair (with extensions that do not necessarily satisfy \eqref{eq:extNEP}). With the similarity transformation 
    \[
    Z=\begin{bmatrix}
    I&-X^Hv\\
    0& 1
    \end{bmatrix}\in\CC^{(p+1)\times(p+1)}
    \]
    we can, by using Lemma~\ref{lem:effenberger}, verify 
    that \eqref{eq:extNEP} is satisfied by selecting
    vectors corresponding $(v,u,\lambda)$ as
    $(\hat{X}Ze_{p+1},  [I_p\;\; 0]Z^{-1}\hat{S}Ze_{p+1}, \lambda)$.
    The converse holds due to the fact that a solution
    satisfying \eqref{eq:extNEP} forms a vector $v$ which is
    orthogonal to $X$, and non-zero since the extension
    would otherwise be non-minimal.
%
%
%
  \end{proof}
\begin{remark}[Minimality index greater than one]\label{rem:ind2}
  The generalization of the above reasoning
  to a   higher minimality index can be seen as follows.
  The orthogonality condition \eqref{eq:orthogonality} with $\ell=2$
  implies that   $v,u,\lambda$ must satisfy
  \[
   (X^H+\lambda S^HX^H)v+S^HX^HXu=0.
  \]
  Unlike the case $\ell=1$, this expression depends on both $u$ and $\lambda$.
  As pointed out in a more general form in \cite{Effenberger:2013:ROBUST},
  the analogous NEP to \eqref{eq:extNEP} becomes
  \[
  \begin{bmatrix}
    M(\lambda)&U(\lambda)   \\
    X^H+\lambda S^HX^H& S^HX^HX
  \end{bmatrix}\begin{bmatrix}
    v\\
    u  \end{bmatrix}=0.
  \]
  Unfortunately, when we include a normalization condition as in \eqref{eq:F0},
  this problem does not lead to a nonlinear equation of the form
  \eqref{eq:structF} which we need for structured Broyden method.
  It includes more blocks and more $\lambda$-dependence,
  \[
0=
  \begin{bmatrix}
    M(\lambda)&U(\lambda)\\
    C_1^H(\lambda)& C_2^H\\
    c_1^H & 0
  \end{bmatrix}\begin{bmatrix}
    v\\
    u
  \end{bmatrix}-
  b\in\CC^{n+p+1},
  \]
  where $C_1^H(\lambda)=X^H+\lambda S^HX^H$ and $C_2^H=S^HX^HX$.
  This prevents us from using structured Broyden in the same way.
  We can however apply Broydens method (without exploiting
  the same amount of structure), which we illustrate in the
  simulations in Section~\ref{sec:simulations}. 
  In this sense, our algorithm presented in the next section 
  can in principle be constructed with a higher minimality index,
  but we cannot use the same amount of structure.
  In this paper we develop an efficient algorithm for
  $\ell=1$, and propose to use the slower variant without structure
  exploitation for  problems where eigenpairs share eigenvectors.
  For most problems
  stemming from PDEs, $\ell=1$ is generic. 
  
\end{remark}

%

\subsection{Structured Broyden with deflation}

The previous section showed that given an
index one invariant pair, we can compute an extension of
that invariant pair by solving the NEP \eqref{eq:MNEP}.
All extensions are represented by this extended
NEP according to Theorem~\ref{thm:equiv}. Since \eqref{eq:MNEP} combined
with the normalization condition with $c^H=[c_1^H\;\;0]$
leads to a nonlinear system of equation of the structure \eqref{eq:structF}
and we can use Algorithm~\ref{alg:structbroyden}  to solve it.

This structured extension of the invariant pair can be combined with
Algorithm~\ref{alg:structbroyden}.  Algorithm~\ref{alg:deflated}
shows this combination, including handling of
invariant pairs and starting values. We now provide
further details and justification of the algorithm,
and show how restarting can be incorporated.

Recall that our method is mainly intended for problems where
the matrix vector product $M(\lambda)z$ is computationally
expensive.
At step 8 of Algorithm~\ref{alg:structbroyden} we
need to compute the residual \eqref{eq:rkdef} which
in our setting contains terms $M(\lambda)v$ and $U(\lambda)u$;
each of these involving one matrix vector product with $M(\lambda)$.
When we use Algorithm~\ref{alg:structbroyden}
we can combine $M$ with the formula for $U$ in \eqref{eq:Udef2} and
compute the residual \eqref{eq:rkdef} directly by using only matrix vector product.
\[
  r_{k+1}=M(\lambda_{k+1})v_{k+1}+U(\lambda_{k+1})u_{k+1}=  M(\lambda_{k+1})(v_{k+1}+X(\lambda_{k+1} I -S)^{-1}u_{k+1}).
  \]

%
%
  Our algorithm requires starting values for each extension of the invariant pair. 
  Although starting values are usually tuned to the applications,
  and this can also be done in our case, we here propose
  a quite general application-independent procedure to select starting values.
  We base the starting values on previously computed information,
  which can  be viewed as a restarting procedure. 
  Starting values are required for $M_1=T_1^{-1}$,
  $W_1$, $v_1,u_1$ and $\lambda_1$. If we are interested in eigenvalues close to a target $\sigma$,
  we propose (Step 1) to use $M_1\approx M(\sigma)$
  (or $M_1=M(\sigma)$ if it can be computed cheaply) and set $\lambda_1=\sigma$.

  The eigenvector approximation ($v_1\in\CC^n$ and $u_1\in\CC^p$) are computed following an approximation of one
  step of  the  method called safeguarded iteration \cite[Algorithm~4]{Mehrmann:2004:NLEVP} in
  Step~6.
  Eigenvector approximations in safeguarded iteration are extracted by selecting the eigenvector
  corresponding to a small eigenvalue of the matrix $M(\lambda)$.
  We select $v_1$ and $u_1$ in this way but applied to the extended deflated NEP \eqref{eq:MNEP},
  by replacing the blocks of the matrix with approximations,
  $M(\lambda_1)\approx M_1$ and $U(\lambda_1)\approx U_0$, where
  $U_0$ is computed directly from  \eqref{eq:Udef2} by using 
  $p$ matrix vector products.

  We see by comparing \eqref{eq:JkstructT}  and \eqref{eq:Jstruct} that
  $W_0$ should be an approximation of $[U(\lambda_0)\;\; M'(\lambda_0)v_0+U'(\lambda_0)u_0]$.
  The approximation of $U(\lambda_0)$ is chosen as the already computed $U_0$.
  The formula for  $U(\lambda)$ in \eqref{eq:Udef2} gives us directly that
\begin{equation}  \label{eq:Uprime}
  U'(\lambda)=-M(\lambda)X(\lambda I-S)^{-2}+M'(\lambda)X(\lambda I-S)^{-1}.
\end{equation}
  In order to compute a starting value of the last column of $W_0$ we use
  that the chain-rule for differentiation applied to $U(\lambda)$ implies
  \begin{multline*}
  M'(\lambda_1)v_1+U'(\lambda_1)u_1=
  M'(\lambda_1)(v_1+X(\lambda_1 I-S)^{-1}u_1)-M(\lambda_1)X(\lambda_1 I-S)^{-2}u_1
  =\\ M'(\lambda_1)(v_1+X(\lambda_1 I-S)^{-1}u_1)-U(\lambda_1)(\lambda_1 I-S)^{-1}u_1.
  \end{multline*}
  Unless the matrix vector action of the derivative $M$ is explicitly available,
  the first term can be approximated by central finite difference, and the second term
  by using the already computed $U_1$, i.e., $U(\lambda_1)(\lambda_1 I-S)^{-1}u_1\approx U_1(\lambda_1 I-S)^{-1}u_1$. This is done in step 8.

%

In Step 11 we expand the invariant pair again if the problem
exhibits symmetry. It is straightforward to show that if
$\overline{M(\lambda)}=M(\overline{\lambda})$ for all $\lambda$,
then an eigenpair $(v,\lambda)$ implies that
$(\overline{v},\overline{\lambda})$ is an eigenpair which
can be included in the invariant pair if $\lambda\not\in\RR$.
The new complex conjugate  pair $(\overline{v},\overline{\lambda})$
is included by carrying out a Gram-Schmidt orthogonalization against $X$, and
storing the Gram-Schmidt coefficients in the new column of $S$.

\renewcommand{\algorithmicrequire}{\textbf{Input:}}
\renewcommand{\algorithmicensure}{\textbf{Output:}}
\begin{algorithm}
\begin{algorithmic}[1]
  \REQUIRE Target $\sigma$ and normalization vector $c\in\CC^{n}\backslash\{0\}$
\ENSURE A standard minimal invariant pair $(X,S)\in\CC^{n\times p}\times\CC^{p\times p}$ of \eqref{eq:NEP}
\STATE Compute $M_1\approx M(\sigma)$ and $T_1\approx M(\sigma)^{-1}$
\STATE Set $X=$empty matrix and $S=$empty matrix
\STATE Set $k=1$
\WHILE{$k<p$}
\STATE Compute $U_1\approx U(\sigma)\in\CC^{n\times (k-1)}$ where $U$ is given by \eqref{eq:Udef2} 
\STATE Compute the smallest (in modulus) eigenvalue  of the matrix
\[
\begin{bmatrix}
  M_1&U_1\\
  X^H& 0
\end{bmatrix}\in\CC^{(n+k-1)\times(n+k-1)}
\]
and let $\begin{bmatrix}v_1^T&u_1^T\end{bmatrix}^T$
  be the corresponding eigenvector normalized such that $c^Hv_1=1$.
  \STATE Impose orthogonalization \eqref{thm:structure} on $v_1$ by updating $v_1$ and $u_1$. 
  \STATE Compute $f_1\approx M'(\sigma)v_1-U_1(\sigma I-S)^{-1}u_1$, e.g. with finite difference for $M'(\sigma)v_1$. 
  \STATE Run structured Broyden for NEPs (Algorithm~1) with $C=\begin{bmatrix}X& c\end{bmatrix}$ starting value $(\sigma,v_1,u_1)$ and
  Jacobian approximation $(T_1,W_1)=(T_1,[U_1\;\;f_1])$. Save output in $(\lambda,v,u)$ and $(T_N,[U_N\;\;f_N])$
  \STATE Expand invariant pair according to \eqref{eq:extendedpair}.
  \STATE $k=k+1$
  \STATE If NEP has conjugate pair symmetry, expand also with conjugate eigenpair.
  \ENDWHILE
\end{algorithmic}
\caption{Deflated Broyden's method\label{alg:deflated}}
\end{algorithm}
%

%

\section{Convergence theory}\label{sec:convergence}
Due to its equivalence with Broyden's method,
the convergence of our approach can be characterized with
more general results. In particular,
Broyden's method has asymptotic local
superlinear convergence in general \cite{Gay:1979:BROYDENCONV,Broyden:1973:BDM}. 
However, the theory for superlinear convergence only holds under
the assumption that the Jacobian at the solution is
invertible. If this is not satisfied you can invoke theory for
Broyden's method of singular Jacobians \cite{Decker:2985:BROYDENSINGULAR},
which implies (in general) linear
convergence with a convergence factor equal to the reciprocal golden ratio.
We characterize the singularity of the Jacobian
of our particular problem.

The singularity
of the Jacobian of many iterative methods
for NEPs are directly given from the multiplicity
(or Jordan chain structure) of the solution to the NEP, cf. \cite{Szyld:2015:LOCALI,Jarlebring:2011:RESINVCONV,Jarlebring:2012:CONVFACT,Szyld:2015:LOCALII}.
Our construction is equivalent to applying Broyden's method
to the augmented NEP \eqref{eq:MNEP}. 
Therefore, the Jacobian singularity of the augmented system \eqref{eq:structF}
can be
characterized with the multiplicity of the augmented
NEP \eqref{eq:MNEP}.
Since the augmented problem is an artificially constructed NEP,
we find it more insightful to characterize the 
Jacobian singularity in terms of the
eigenvalue multiplicity of the original NEP \eqref{eq:NEP}.

We note that the Jacobian of \eqref{eq:structF}
is given by
\begin{equation}  \label{eq:Jacobian}
  J_*=\begin{bmatrix}
  M(\lambda)&U(\lambda) &M'(\lambda)v+U'(\lambda)u\\
  X^H&      &     \\
  c^H&      &   
\end{bmatrix}
\end{equation}
and provide two convergence results.
It turns out that the condition that the vector $c$ should not be orthogonal to the
eigenvector generalizes to the condition
that matrix
\begin{equation}\label{eq:orthcond_jordan}
\begin{bmatrix}
  X^H\\ c^H
\end{bmatrix}\begin{bmatrix}
  X&v_1\\
\end{bmatrix}
\end{equation}
needs to be non-singular, which is needed in the following theorem which gives a
precise condition for the Jacobian to be singular.

\begin{theorem}[Jacobian singularity]\label{thm:simple}
  Suppose $(X,S)$ is a
  minimal index one invariant pair of \eqref{eq:NEP},
  where $X$ is orthogonal and $S$ upper triangular.
  Suppose $\lambda_1,v_1$ is an eigenpair of \eqref{eq:NEP}
  such that 
  \eqref{eq:orthcond_jordan} is non-singular,
  and suppose $M(\lambda_1)$ has null space of dimension one.
  Moreover, assume  $\lambda_1\not\in\lambda(S)$ and
  $v_1\not\in\operatorname{Range}(X)$.
  Then
  \begin{equation}  \label{eq:vtransformation}
   \begin{bmatrix}
    v\\u\\\lambda
  \end{bmatrix}=
  \begin{bmatrix}
    (I-XX^H)v_1\\
    (\lambda_1I-S)X^Hv_1\\
    \lambda_1
  \end{bmatrix}
  \end{equation}
  is a solution to \eqref{eq:extNEP}.
  Moreover, the Jacobian \eqref{eq:Jacobian} corresponding
  to this solution is singular if and only if there exists
  a Jordan chain of length two, i.e., there
  exists a vector $v_2$ such that
  \begin{equation}  \label{eq:jordanlength_two}    
    M(\lambda_1)v_2+M'(\lambda_1)v_1=0.
  \end{equation}
  \end{theorem}
\begin{proof}
We verify \eqref{eq:vtransformation} directly by inserting into \eqref{eq:extNEP} and using
the formula for  $U$ in \eqref{eq:Udef2} and that $M(\lambda_1)v_1=0$.
In order to establish when the Jacobian is singular we give
necessary and sufficient conditions for
the existence of non-trivial $z_1,z_2,z_3$ such that
\begin{equation}\label{eq:singvect}
\begin{bmatrix}
  M(\lambda)&U(\lambda) &M'(\lambda)v+U'(\lambda)u\\
  X^H&      &     \\
  c^H&      &   
\end{bmatrix}\begin{bmatrix}
  z_1\\
  z_2\\
  z_3
\end{bmatrix}=0.
\end{equation}
By using the formula for $U'(\lambda_1)$ in \eqref{eq:Uprime}
and the formula for the solution vector \eqref{eq:vtransformation},
the first block equation becomes
\begin{equation}  \label{eq:singvect_firstrow}
 M(\lambda_1)(z_1+X(\lambda I-S)^{-1}z_2+X(\lambda I-S)^{-1}X^Hv_1z_3)
+M'(\lambda_1)v_1z_3=0
\end{equation}
We separate the rest of the proof into two cases.
\begin{itemize}
\item Suppose $z_3=0$, such that \eqref{eq:singvect_firstrow}
  reduces to $M(\lambda)(z_1+X(\lambda I-S)^{-1})z_2)=0$.
  Since $M(\lambda_1)$ has a one-dimensional null space, we
must have $z_1+X(\lambda I-S)^{-1})z_2=\beta v_1$.
By multiplication from left with $\begin{bmatrix}X&c\end{bmatrix}^H$,
  combining this with the last rows in \eqref{eq:singvect} 
  and using the assumption \eqref{eq:orthcond_jordan} we see that  $\beta=0$
  and $z_2=0$. Consequently, $z_1=0$, such that $z_1,z_2,z_3$ are
  identically zero and do  not
  form a non-trivial singular vector. Hence,
  any non-trivial singular vector must satisfy $z_3\neq 0$.
\item If we assume that $z_3\neq 0$, we can without loss of generality
  assume that $z_3=1$.
  Clearly \eqref{eq:singvect_firstrow} can only be zero if there exists
  a vector $v_2$ such that \eqref{eq:jordanlength_two} is satisfied.
Moreover, $z_1$
and $z_2$ must satisfy for some value $\beta$, 
\[
z_1+X(\lambda I-S)^{-1}z_2+X(\lambda I-S)^{-1}X^Hv_1=v_2+\beta v_1
\]
We obtain that
\[
\left(\begin{bmatrix}
  X^H\\ c^H
\end{bmatrix}\begin{bmatrix}
  X&v_1\\
\end{bmatrix}\right) \begin{bmatrix}
  (\lambda I-S)^{-1}z_2\\
-\beta  
\end{bmatrix}=
\begin{bmatrix}
  X^H\\ c^H
\end{bmatrix}\left(-X(\lambda I-S)^{-1}X^Hv_1+v_2\right)
\]
This linear system has a solution since \eqref{eq:orthcond_jordan}
is invertible by assumption, and
directly gives us a singular vector from a vector $v_2$ satisfying \eqref{eq:jordanlength_two}.
It is non-trivial since $z_3=1$.
\end{itemize}

%
\end{proof}

  \begin{example}[Double eigenvalue]\rm
    The convergence properties for a singular Jacobian matrix
    can be observed in practice, and we illustrate this with the NEP  
    presented in \cite{Jarlebring:2010:INVARIANCE}
    (and also \cite{Jarlebring:2012:CONVFACT,Jarlebring:2011:RESINVCONV}),
    The problem is a delay eigenvalue problem $M(\lambda)=-\lambda I+A_0+A_1e^{-\tau\lambda}$
    constructed such that it
    has a double non-semisimple
    eigenvalue at $\lambda_*=3\pi i$.
    The error history of Algorithm~\ref{alg:deflated} with
    $\sigma=0$ and $M_0=M(\sigma)$ is given in Figure~\ref{fig:double_eig_conv}.
    We clearly see 
    that we have linear convergence the first time the iteration converges
    to $\lambda_*$. The second time the iteration converges to
    $\lambda_*$ we have superlinear convergence, consistent with the fact that
    the eigenvalue has multiplicity two (and not three).
    Once one of the double eigenvalues has been deflated,
    the Jacobian is singular, i.e.,  the convergence   behaves as the
    convergence  for simple eigenvalue.

%
    %

    We also observe (consistent  with theory \cite{Decker:2985:BROYDENSINGULAR})
that the linear convergence has convergence factor equal to the reciprocal golden ratio, i.e., approximmately $0.618$.
    
    \end{example}
  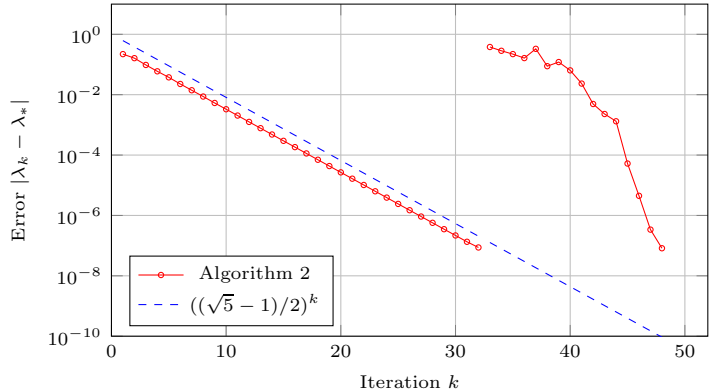
\begin{figure}[h]
    \begin{center}
       \input{double_eig_conv.tex}
       \caption{Convergence for the problem with a double eigenvalue at $\lambda_*=3\pi i$.
         We clearly observe linear convergence the first time the iteration
         converges to $\lambda_*$, and fast superlinear convergence the second time it
         converges to $\lambda_*$.
        \label{fig:double_eig_conv}
      }
    \end{center}
  \end{figure}
  \section{Simulations for quadratic time-delay system}\label{sec:simulations}
\begin{figure}[h]
  \begin{center}
    \subfigure[Error vs iteration]{\input{num_comp1_iter.tex}}
    \subfigure[Error vs wall time]{\input{num_comp1_time.tex}}
    \caption{Comparison of structure exploiting and non-exploiting Broyden's method.
      \label{fig:num_comp1}
    }
  \end{center}
\end{figure}
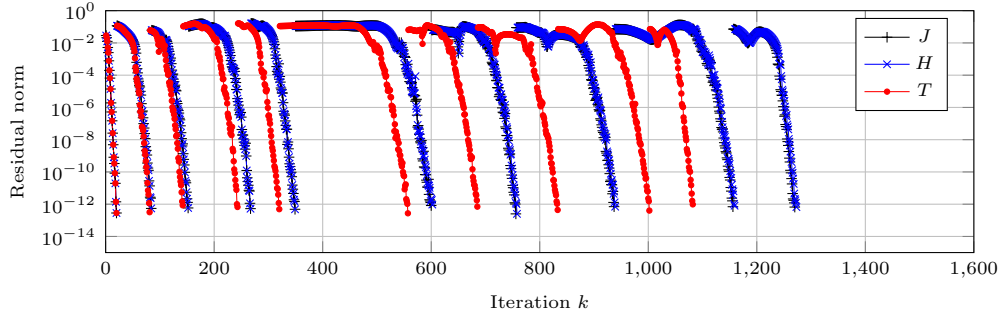
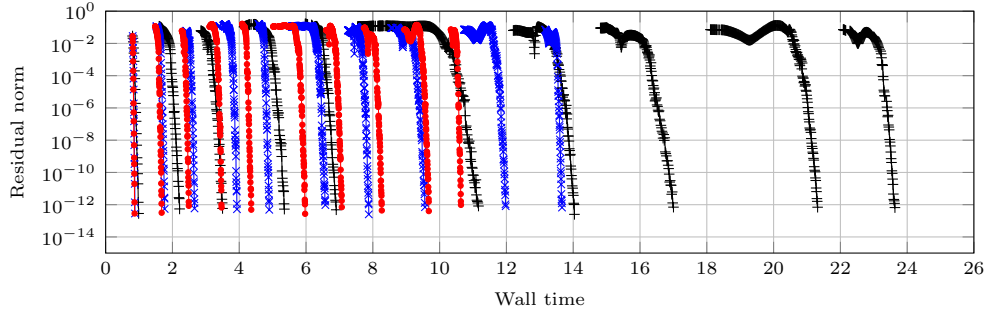
\begin{figure}[h]
  \begin{center}
\input{num_comp_ell2.tex}
\caption{Comparison with minimality index one or two.
  \label{fig:num_comp_ell2}
    }
  \end{center}
\end{figure}
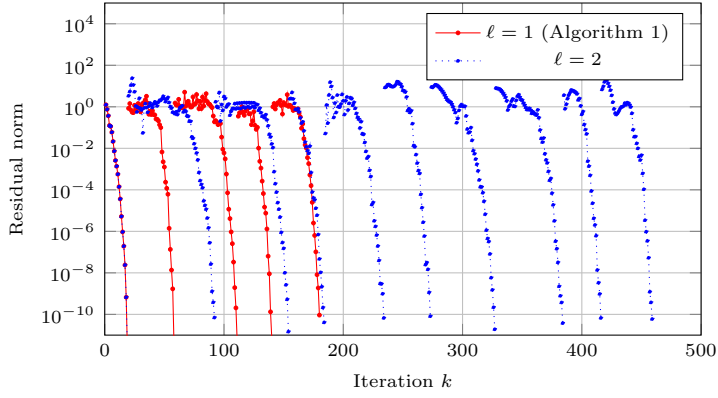

We provide results of simulations for various NEPs. Our
implementation is in the Julia programming language \cite{Bezanson:2017:JULIA}, version 0.6.2 using a quad-core, 16 GB RAM, Intel i7-4600U CPU with 2.10GHz\footnote{The simulations are publicly available
  online: \url{http://www.math.kth.se/~eliasj/src/broyden}}.

In order to show the properties of our approach we apply now
apply the algorithm to the following problem
\[ 
  M(\lambda)=-\lambda^2 I+ A_0+A_1e^{-\lambda}
\]
where the matrices are the same as those in \cite{Gaaf:2017:INFBILANCZOS}. 
The simulations of this section are intended to illustrate
method properties, and we do not claim that this method
is the best method for this type of problem.

We first illustrate the structure exploiation. In Figure~\ref{fig:num_comp1}
we see the convergence of the discussed versions of Broyden's method.
Figure~\ref{fig:num_comp1}a and Figure~\ref{fig:num_comp1}b
show the same simulation but with different $x$-axis.
The structure exploiting Broyden method (Algorithm~\ref{alg:structbroyden})
converges (slightly) faster in terms of iterations, although they are equivalent
in exact arithmetic. The structure exploiting Broyden method
is considerably faster than the other variants in terms of computation time.

The relevance of the damping is illustrated in Figure~\ref{fig:num_comp2}.
No damping (or a very large $t$) typically leads to faster convergence,
but robustness is lost as the solution can start diverging.
The parameter $t$ can be viewed as a 
trade-off parameter,  between robustness
and convergence speed.

In order to illustrate the value of superlinear convergence, we
compare the algorithm residual inverse iteration as described in \cite{Neumaier:1985:RESINV}, which is
a very well established method.
Residual inverse iteration is an implicit quasi-Newton method \cite{Jarlebring:2017:QUASINEWTON}
and exhibits linear convergence. We see in Figure~\ref{fig:num_comp3}
that our the proposed method is faster in terms of iterations.
In the residual inverse iteration
we have pre-computed an LU-factorization,
in order speed up the computation of the linear solves.

In order to illustrate that a higher minimality index
can allow you to compute more than $n$ eigenvalues,
we adapted to idea described in Remark~\ref{rem:ind2}.
A comparison with $\ell=2$ can be seen in
Figure~\ref{fig:num_comp_ell2} with $n=5$. 
Minimality index $\ell=2$ provides the
possibility to compute $2n=10$ eigenvalues.

%









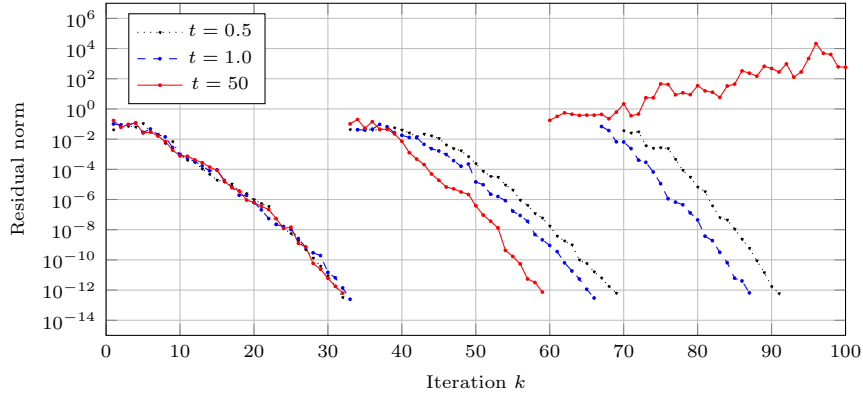
\begin{figure}[h]
  \begin{center}
\input{num_comp2_iter.tex}
    \caption{Comparison for different threshold parameters
      \label{fig:num_comp2}
    }
  \end{center}
\end{figure}

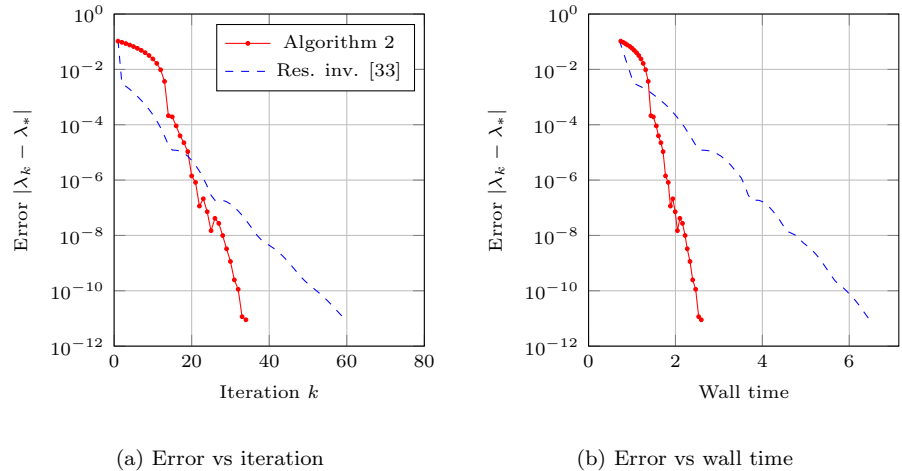
\begin{figure}[h]
  \begin{center}
    \subfigure[Error vs iteration]{\input{num_comp3_iter.tex}}
    \subfigure[Error vs wall time]{\input{num_comp3_time.tex}}    
    \caption{Comparison with residual inverse iteration
      \label{fig:num_comp3}
    }
  \end{center}
\end{figure}

\section{Simulations for time-periodic delay-differential equation}\label{sec:TPDDE}
The following problem is called time-periodic delay-differential
equation. We consider a linear (time-varying)
delay-differential equation
\begin{equation}  \label{eq:PDDE}
   \dot{y}(t)=A(t)y(t)+B(t)y(t-\tau)  
\end{equation}
where $A(t),B(t)\in\CC^{n\times n}$ are periodic functions
with period $\tau$.
We briefly summarize a  stability characterization which leads to a NEP.
See description of
certain applications \cite{MANN200335}
and references therein and a number of 
numerical methods 
\cite{Szalai:2006:CONTINUATION}
\cite{Insberger:2002:MATHIEU} 
   \cite{Insperger:2004:UPDATEDSEMIDISC}
   \cite{Insperger:2002:PERIODICDDES} 
   \cite{Breda:2006:PERIODIC}   for details.
   The observervation that \eqref{eq:PDDE}
   can be characterized with a  NEP
   was also used in \cite{Rott:2010:ITERATIVE}.
We consider the ODE (without delay) associated with \eqref{eq:PDDE}
   \begin{equation}  \label{eq:PODE}
       \dot{p}(t)=C(t,\lambda)p(t),
   \end{equation}
   where
   \[
   C(t,\lambda):=A(t)+B(t)e^{-\lambda \tau}-\lambda I.
   \]
   We define a NEP $M(\lambda)$ by the action on a vector as 
   \begin{equation}  \label{eq:PDDE_NEP}     
    M(\lambda)v=p(\tau)-v
   \end{equation}
    where $p(\tau)$ is the solution $p(t)$ of \eqref{eq:PODE}
    at $t=\tau$ with initial condition
    \[
   p(0)=v.
   \]
   The solutions of the NEP defined by \eqref{eq:PDDE_NEP},
   correspond to  $(\lambda,v)$ such that  $p(\tau)=v=p(0)$, i.e.,
   the starting value and
   final vector of $p$ are the same and $p$ can be viewed as a
   periodic function.  From Floquet theory one can
   show that the stability of \eqref{eq:PDDE}
   is determined from the right-most solution $\lambda$.
The value $\mu=e^{\tau\lambda}$ is
   called the characteristic multiplier, which is greater than
   one for right-half plane solutions to the NEP.

   Note that the NEP given by \eqref{eq:PDDE_NEP}, has an action defined
   by a solution to an ODE, i.e., the action is compuationally
   expensive and it is 
   of the type we consider in this work.

   \subsection{Benchmark problem}\label{sec:milling1} Time-periodic
   time-delay systems has been considerably used in models
   and studied in for specific applications in the literature.
   Certain vibrations 
   in machine tool milling can be modeled with time-periodic
   time-delay systems, where dominant modes correspond
   to the undesirable machine tool chatter. The delay in this case stems from the
   fact that the cut of the previous lap has an influence
   on the current lap. The periodicity stems from the periodicity
   in the force, and modeling of the cutting tooth which is periodic
   in time due to the rotation.
   We consider a specific setup used as a benchmark in
   several papers. See  \cite{Insberger:2002:MATHIEU} 
   \cite{Insperger:2004:UPDATEDSEMIDISC}
   \cite{Insperger:2002:PERIODICDDES} and references therein.
   The equations of motion are second order but can be
   reformulated into a first order time-periodic time-delay system
   \[
\dot{y}=
\begin{bmatrix}
  0&1\\
   -\omega_0^2-\frac{a_pw(t)}{m}&-2\zeta\omega_0
\end{bmatrix}
y(t)+\\
\begin{bmatrix}
  0&0\\
\frac{a_pw(t)}{m}&0
\end{bmatrix}y(t-\tau)
\]
By consideration of the projection of the application
of the force (as described in \cite{Rott:2010:ITERATIVE}),
the time-periodic coefficient becomes 
\[
 w(t)=H(t-\tau/2)(\sin^2(\phi(t))K_R+\cos(\phi(t))\sin(\phi(t))K_T)
 \]
where $H(t)$ is the heaviside function and $\phi(t)=2\pi t /\tau$.
We see that if the force modeling is not considered, $w(t)$ is
constant the problem reduces to the (easier) standard time-delay system.

We carried out simulations with parameters $a_p=m=\tau=\omega_0=\zeta=1$,
and solved the time-dependent ODE with Runge-Kutta 4 with $N$ discretization
points.
The convergence as a function of iteration is given in Figure~\ref{fig:milling1}b.

One of the most successful numerical approaches
for this problem correspond to discretizations
of operator formulations, e.g., a spectral
discretization of the monodromy operator in
\cite{Bueler:2007:ERROR} and \cite{Breda:2006:PERIODIC}.
In practice, this involves the solution of a large (linear)
eigenvalue problem. 
A comparison with the approach in \cite{Bueler:2007:ERROR} is
shown in Figure~\ref{fig:milling1}a. We clearly see that 
the discretization in \cite{Bueler:2007:ERROR}
and our approach  lead to  algebraic convergence, of similar order.
That is, although a spectral discretization is
used in \cite{Bueler:2007:ERROR}, the observed convergence
with respect to ODE-discretization is not exponential, but
only algebraic. This is expected since the $A(t)$
has a discontinuous derivative, and one cannot in general
expect exponential convergence for PDEs which
have discontinuous derivatives.

\begin{figure}[h]
  \begin{center}
    \subfigure[Continuous problem error vs time-discretization]{\input{milling1_disc.tex}}
    \subfigure[Continuous problem error of continuous vs iteration]{\input{milling1_iter.tex}}
    \caption{Simulation with benchmark problem in Section~\ref{sec:milling1}
      for different time-discretizations
      \label{fig:milling1}
    }
  \end{center}
\end{figure}

\subsection{Benchmark problem with  PDE coupling}\label{sec:milling}
In order to also take into account vibrations in
the workpiece in the milling, a model which couples a PDE 
was presented in \cite{Rott:2010:ITERATIVE}. A discretization
of the PDE leads to the following problem.
Let $D_{xx}=\frac{1}{h}\operatorname{tridiag}(1,-2,1)
\in\RR^{N\times N}$
where $h=1/N$.
The identity operator in the finite-element basis is denoted  $P^{-1}$
and $p_n=Pe_n$.
The time-periodic time-delay system is now given by 
\begin{multline*}
\dot{y}(t)=
\begin{bmatrix}
  & &I&\\
  & &&1\\
  -\epsilon PD_{xx}-\frac{a_pw(t)}{A}p_Ne_N^T&   -\frac{a_pw(t)}{A}e_N&-dPD_{xx}&\\
   -\frac{a_pw(t)}{m}e_N^T &-\omega_0^2-\frac{a_pw(t)}{m}&&-2\zeta\omega_0
\end{bmatrix}
y(t)+\\
\begin{bmatrix}
  & &&\\
  & &&\\
  \frac{a_pw(t)}{A}p_Ne_N^T&\frac{a_pw(t)}{A}p_N&\phantom{0}&\\
   \frac{a_pw(t)}{m}e_N^T&\frac{a_pw(t)}{m}&&
\end{bmatrix}y(t-\tau).
\end{multline*}
We carried out simulations for a discretization with  $N=5000$, i.e., $n=10002$ on a
computer with 64 GB of RAM. The results
are presented in Figure~\ref{fig:big}.  The action of the ODE was discretized with $N=15$,
whereas the the approximation of $M(\sigma)$ was computed with $N=7$.
The problem is stiff, and we therefore used an implicit time-stepping scheme.
The inverse of the identity was treated in a way that avoids computing
a full matrix. This problem is of such size that our implementation of 
the approach of \cite{Bueler:2007:ERROR} was not applicable due
to the high demand of memory resources.
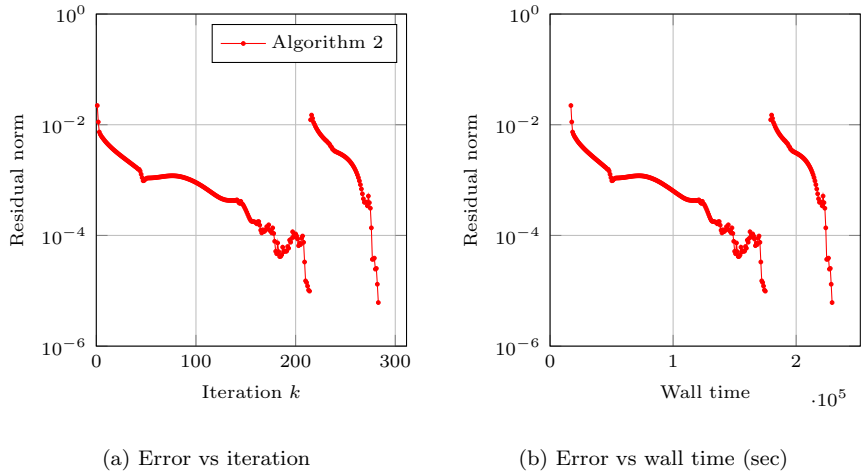
\begin{figure}[h]
  \begin{center}
    \subfigure[Error vs iteration]{\input{big_iter.tex}}
    \subfigure[Error vs wall time (sec)]{\input{big_time.tex}}    
    \caption{PDE-coupling milling simulation
      \label{fig:big}
    }
  \end{center}
\end{figure}

\section{Conclusions and outlook}

Broyden's method, a standard approach for nonlinear systems of equations,
has here been developed to and turned into a useful algorithnm
for certain types of NEPs. Broyden's method has been developed and
studied considerable in the literature. Several techniques seem to carry over directly, such as limited memory versions \cite{Deuflard:2004:NEWTON,Rotten:2015:BROYDEN}, or variations, such as the (so-called) bad Broyden's method \cite{Broyden:1965:BROYDEN} can be specialized completely analogous to our approach.
In order to maintain generality we have intentionally not pursued a detailed
study of the ODE-solver used in the  specific application in Section~\ref{sec:milling}. A more specialized result using the structure of the  matrices
would probably lead to even further efficiency, but would be
beyond the scope of this paper about methods for NEPs rather than the specific problem Section~\ref{sec:milling}. 


%
%
%
%
\section*{Acknowledgment}
The author is grateful for the valuable discussions
about Broyden's method with David Ek  and Anders Forsgren
of the mathematics department, KTH Royal institute of technology.

\bibliographystyle{plain}
\bibliography{eliasbib,misc}

%
%
\end{document}

%% file: roundoff_conv1.tex
\setlength\figureheight{6cm} 
\setlength\figurewidth{7cm}
\begin{tikzpicture}
\begin{axis}[%
width=0.951\figurewidth,
height=\figureheight,
xmin=0,
xmax=50,
xlabel={Iteration $k$},
ymode=log,
ymin=1e-7,
ytick={1e2,1,1e-2,1e-4,1e-6,1e-8,1e-10,1e-12,1e-14,1e-16,1e-18,1e-22,1e-26,1e-30,1e-34},
ymax=100,
ylabel={Residual norm},
grid,
legend pos=south west,
unbounded coords=jump
]
  \addlegendentry{Exact};
  \addplot[    mark=square, only marks, 
    mark options={solid},
    mark size=2.5pt,
    ] table [x index=0, y index=1, col sep=comma]
   {roundoff_conv_data.csv};
 \addlegendentry{$J$};
  \addplot[    mark=+, 
    mark options={solid},
    mark size=2.0pt,
    ] table [x index=0, y index=2, col sep=comma]
   {roundoff_conv_data.csv};
 \addlegendentry{$H$};
  \addplot[    mark=x, 
    mark options={solid},
    mark size=2.0pt, blue
    ] table [x index=0, y index=3, col sep=comma]
   {roundoff_conv_data.csv};
 \addlegendentry{$T$};
  \addplot[    mark=*, red,
    mark options={fill=red},
    mark size=1.0pt,
    ] table [x index=0, y index=4, col sep=comma]
   {roundoff_conv_data.csv};

\end{axis}
\end{tikzpicture}

%% file: roundoff_conv.tex
\setlength\figureheight{6cm} 
\setlength\figurewidth{6cm}
\begin{tikzpicture}
\begin{axis}[%
width=0.951\figurewidth,
height=\figureheight,
xmin=40,
xmax=50,
xlabel={Iteration $k$},
ymode=log,
ymin=1e-7,
ytick={1e2,1,1e-1,1e-2,1e-3,1e-4,1e-6,1e-8,1e-10,1e-12,1e-14,1e-16,1e-18,1e-22,1e-26,1e-30,1e-34},
ymax=1e-3,
ylabel={Residual norm},
grid,
legend pos=south west,
unbounded coords=jump
]
  \addlegendentry{Exact};
  \addplot[    mark=square, only marks, 
    mark options={solid},
    mark size=2.5pt,
    ] table [x index=0, y index=1, col sep=comma]
   {roundoff_conv_data.csv};
 \addlegendentry{$J$};
  \addplot[    mark=+, 
    mark options={solid},
    mark size=2.0pt,
    ] table [x index=0, y index=2, col sep=comma]
   {roundoff_conv_data.csv};
 \addlegendentry{$H$};
  \addplot[    mark=x, 
    mark options={solid},
    mark size=2.0pt, blue
    ] table [x index=0, y index=3, col sep=comma]
   {roundoff_conv_data.csv};
 \addlegendentry{$T$};
  \addplot[    mark=*, red,
    mark size=1.0pt,
    mark options={fill=red},
    ] table [x index=0, y index=4, col sep=comma]
   {roundoff_conv_data.csv};

\end{axis}
\end{tikzpicture}

%% file: double_eig_conv.tex
\tikzsetnextfilename{myfig1} 

\setlength\figureheight{6cm} 
\setlength\figurewidth{10cm}
\begin{tikzpicture}
\begin{axis}[%
width=0.951\figurewidth,
height=\figureheight,
xmin=0,
xmax=52,
xlabel={Iteration $k$},
ymode=log,
ymin=1e-10,
ytick={1e2,1,1e-2,1e-4,1e-6,1e-8,1e-10,1e-12,1e-14,1e-16,1e-18,1e-22,1e-26,1e-30,1e-34},
ymax=10,
ylabel={Error $|\lambda_k-\lambda_*|$},
grid,
legend pos=south west,
unbounded coords=jump
]
\addlegendentry{Algorithm~2};
  \addplot[    mark=o, red,
    mark options={fill=red},
    mark size=1.0pt,
    ] table [x index=0, y index=1, col sep=comma]
   {double_eig_conv_data.csv};
\addlegendentry{ $((\sqrt{5}-1)/2)^k$};
\addplot[    blue, dashed
    ] table [x index=0, y index=2, col sep=comma]
   {double_eig_conv_data.csv};
\end{axis}
\end{tikzpicture}

%% file: num_comp1_iter.tex
\setlength\figureheight{6cm} 
\setlength\figurewidth{6cm}
\begin{tikzpicture}
\begin{axis}[%
width=2.3*0.951\figurewidth,
height=0.8\figureheight,
xlabel={Iteration $k$},
ymode=log,
ymin=1e-15,
ytick={1e2,1,1e-2,1e-4,1e-6,1e-8,1e-10,1e-12,1e-14,1e-16,1e-18,1e-22,1e-26,1e-30,1e-34},
ymax=1,
xmax=1600,
xmin=0,
ylabel={Residual norm},
grid,
legend pos=north east,
unbounded coords=jump
]
 \addlegendentry{$J$};
  \addplot[    mark=+, 
    mark options={solid},
    mark size=2.0pt,
    ] table [x index=0, y index=1, col sep=comma]
   {num_comp1_J_data.csv};
 \addlegendentry{$H$};
  \addplot[    mark=x, 
    mark options={solid},
    mark size=2.0pt, blue
    ] table [x index=0, y index=1, col sep=comma]
   {num_comp1_H_data.csv};
 \addlegendentry{$T$};
  \addplot[    mark=*, red,
    mark size=1.0pt,
    mark options={fill=red},
    ] table [x index=0, y index=1, col sep=comma]
   {num_comp1_T_data.csv};

\end{axis}
\end{tikzpicture}

%% file: num_comp1_time.tex
\setlength\figureheight{6cm} 
\setlength\figurewidth{6cm}
\begin{tikzpicture}
\begin{axis}[%
width=2.3*0.951\figurewidth,
height=0.8\figureheight,
xlabel={Wall time},
ymode=log,
ymin=1e-15,
ytick={1e2,1,1e-2,1e-4,1e-6,1e-8,1e-10,1e-12,1e-14,1e-16,1e-18,1e-22,1e-26,1e-30,1e-34},
ymax=1,
xmin=0,
ylabel={Residual norm},
grid,
unbounded coords=jump
]
  \addplot[    mark=+, 
    mark options={solid},
    mark size=2.0pt,
    ] table [x index=2, y index=1, col sep=comma]
   {num_comp1_J_data.csv};
  \addplot[    mark=x, 
    mark options={solid},
    mark size=2.0pt, blue
    ] table [x index=2, y index=1, col sep=comma]
   {num_comp1_H_data.csv};
  \addplot[    mark=*, red,
    mark size=1.0pt,
    mark options={fill=red},
    ] table [x index=2, y index=1, col sep=comma]
   {num_comp1_T_data.csv};

\end{axis}
\end{tikzpicture}

%% file: num_comp_ell2.tex
\setlength\figureheight{6cm} 
\setlength\figurewidth{10cm}
\begin{tikzpicture}
\begin{axis}[%
width=0.951\figurewidth,
height=\figureheight,
xlabel={Iteration $k$},
ymode=log,
ymin=1e-11,
ytick={1e6,1e4,1e2,1,1e-2,1e-4,1e-6,1e-8,1e-10,1e-12,1e-14,1e-16,1e-18,1e-22,1e-26,1e-30,1e-34},
ymax=1e5,
xmin=0,
xmax=500,
ylabel={Residual norm},
grid,
legend pos=north east,
unbounded coords=jump
]
 \addlegendentry{$\ell=1$ (Algorithm~1)};
  \addplot[    mark=*, 
    mark size=0.7pt,
    mark options={fill=red},
    red
    ] table [x index=0, y index=1, col sep=comma]
   {num_comp_ell2_1_data.csv};
 \addlegendentry{$\ell=2$};
  \addplot[    mark=*, 
    mark size=0.7pt,
    mark options={fill=blue},
    blue,dotted
    ] table [x index=0, y index=1, col sep=comma]
   {num_comp_ell2_2_data.csv};
\end{axis}
\end{tikzpicture}

%% file: num_comp2_iter.tex
\setlength\figureheight{6cm} 
\setlength\figurewidth{6cm}
\begin{tikzpicture}
\begin{axis}[%
width=2*0.951\figurewidth,
height=\figureheight,
xlabel={Iteration $k$},
ymode=log,
ymin=1e-15,
ytick={1e6,1e4,1e2,1,1e-2,1e-4,1e-6,1e-8,1e-10,1e-12,1e-14,1e-16,1e-18,1e-22,1e-26,1e-30,1e-34},
ymax=1e7,
xmin=0,
xmax=100,
ylabel={Residual norm},
grid,
legend pos=north west,
unbounded coords=jump
]
 \addlegendentry{$t=0.5$};
  \addplot[    mark=*, 
    mark size=0.5pt,
    mark options={fill=black},
    dotted
    ] table [x index=0, y index=1, col sep=comma]
   {num_comp2_1_data.csv};
 \addlegendentry{$t=1.0$};
  \addplot[    mark=*, 
    mark size=0.5pt,
    mark options={fill=blue},
    blue,
    dashed
    ] table [x index=0, y index=1, col sep=comma]
   {num_comp2_2_data.csv};
 \addlegendentry{$t=50$};
  \addplot[    mark=*, red,
    mark size=0.5pt,
    mark options={fill=red},
    ] table [x index=0, y index=1, col sep=comma]
   {num_comp2_3_data.csv};

\end{axis}
\end{tikzpicture}

%% file: num_comp3_iter.tex
\setlength\figureheight{6cm} 
\setlength\figurewidth{6cm}
\begin{tikzpicture}
\begin{axis}[%
width=0.951\figurewidth,
height=\figureheight,
xlabel={Iteration $k$},
ymode=log,
ymin=1e-12,
ytick={1e6,1e4,1e2,1,1e-2,1e-4,1e-6,1e-8,1e-10,1e-12,1e-14,1e-16,1e-18,1e-22,1e-26,1e-30,1e-34},
ymax=1,
xmin=0,
xmax=80,
ylabel={Error $|\lambda_k-\lambda_*|$},
grid,
legend pos=north east,
unbounded coords=jump
]
 \addlegendentry{Algorithm~2};
  \addplot[    mark=*, 
    mark size=0.7pt,
    mark options={fill=red},
    red
    ] table [x index=0, y index=1, col sep=comma]
   {num_comp3_T_data.csv};
 \addlegendentry{Res. inv. \cite{Neumaier:1985:RESINV}};
  \addplot[    mark=d, 
    mark size=0.7pt,
    mark options={fill=blue},
    blue,
    dashed
    ] table [x index=0, y index=1, col sep=comma]
   {num_comp3_R_data.csv};
\end{axis}
\end{tikzpicture}

%% file: num_comp3_time.tex
\setlength\figureheight{6cm} 
\setlength\figurewidth{6cm}
\begin{tikzpicture}
\begin{axis}[%
width=0.951\figurewidth,
height=\figureheight,
xlabel={Wall time},
ymode=log,
ymin=1e-12,
ytick={1e6,1e4,1e2,1,1e-2,1e-4,1e-6,1e-8,1e-10,1e-12,1e-14,1e-16,1e-18,1e-22,1e-26,1e-30,1e-34},
ymax=1,
xmin=0,
ylabel={Error $|\lambda_k-\lambda_*|$},
grid,
legend pos=north east,
unbounded coords=jump
]
  \addplot[    mark=*, 
    mark size=0.7pt,
    mark options={fill=red},
    red
    ] table [x index=2, y index=1, col sep=comma]
   {num_comp3_T_data.csv};
  \addplot[    mark=d, 
    mark size=0.7pt,
    mark options={fill=blue},
    blue,
    dashed
    ] table [x index=2, y index=1, col sep=comma]
   {num_comp3_R_data.csv};
\end{axis}
\end{tikzpicture}

%% file: milling1_disc.tex
\setlength\figureheight{6cm} 
\setlength\figurewidth{6cm}
\begin{tikzpicture}
\begin{axis}[%
width=0.951\figurewidth,
height=\figureheight,
xlabel={Discretization points (in time)},
ymode=log,
xmode=log,
ymin=1e-9,
ymax=1e3,
xmin=1e-3,
xmax=1e0,
ylabel={Eigenvalue error},
grid,
legend pos=south east,
unbounded coords=jump
]
 \addlegendentry{Algorithm~2 };
  \addplot[    mark=*, 
    mark size=0.7pt,
    mark options={fill=red},
    red
    ] table [x index=0, y index=1, col sep=comma]
   {milling1_data.csv};
 \addlegendentry{Cheb. discr. \cite{Bueler:2007:ERROR}};
  \addplot[    mark=+, 
    mark size=0.7pt,
    blue
    ] table [x index=0, y index=2, col sep=comma]
   {milling1_data.csv};
\end{axis}
\end{tikzpicture}

%% file: milling1_iter.tex
\setlength\figureheight{6cm} 
\setlength\figurewidth{6cm}
\begin{tikzpicture}
\begin{axis}[%
width=0.951\figurewidth,
height=\figureheight,
xlabel={Iteration $k$},
ymode=log,
ymin=1e-11,
ytick={1e6,1e4,1e2,1,1e-2,1e-4,1e-6,1e-8,1e-10,1e-12,1e-14,1e-16,1e-18,1e-22,1e-26,1e-30,1e-34},
ymax=10,
xmin=0,
ylabel={Eigenvalue error},
grid,
legend pos=south west,
unbounded coords=jump
]
 \addlegendentry{$N=10$};
  \addplot[    mark=*, 
    mark size=0.7pt,
    mark options={fill=blue},
    blue,dotted
    ] table [x index=0, y index=1, col sep=comma]
   {milling1_iter1_data.csv};
 \addlegendentry{$N=100$};
  \addplot[    mark=*, 
    mark size=0.7pt,
    mark options={fill=magenta},
    magenta,dashed,
    ] table [x index=0, y index=1, col sep=comma]
   {milling1_iter2_data.csv};
 \addlegendentry{$N=1000$};
  \addplot[    mark=*, 
    mark size=0.7pt,
    mark options={fill=red},
    red,
    ] table [x index=0, y index=1, col sep=comma]
   {milling1_iter3_data.csv};
\end{axis}
\end{tikzpicture}

%% file: big_iter.tex
\setlength\figureheight{6cm} 
\setlength\figurewidth{6cm}
\begin{tikzpicture}
\begin{axis}[%
width=0.951\figurewidth,
height=\figureheight,
xlabel={Iteration $k$},
ymode=log,
ymin=1e-06,
ytick={1e6,1e4,1e2,1,1e-2,1e-4,1e-6,1e-8,1e-10,1e-12,1e-14,1e-16,1e-18,1e-22,1e-26,1e-30,1e-34},
ymax=1,
xmin=0,
ylabel={Residual norm},
grid,
legend pos=north east,
unbounded coords=jump
]
 \addlegendentry{Algorithm~2};
  \addplot[    mark=*, 
    mark size=0.7pt,
    mark options={fill=red},
    red
    ] table [x index=0, y index=1, col sep=comma]
   {big_plot_data.csv};
\end{axis}
\end{tikzpicture}

%% file: big_time.tex
\setlength\figureheight{6cm} 
\setlength\figurewidth{6cm}
\begin{tikzpicture}
\begin{axis}[%
width=0.951\figurewidth,
height=\figureheight,
xlabel={Wall time},
ymode=log,
ymin=1e-06,
ytick={1e6,1e4,1e2,1,1e-2,1e-4,1e-6,1e-8,1e-10,1e-12,1e-14,1e-16,1e-18,1e-22,1e-26,1e-30,1e-34},
ymax=1,
xmin=0,
ylabel={Residual norm},
grid,
legend pos=north east,
unbounded coords=jump
]
  \addplot[    mark=*, 
    mark size=0.7pt,
    mark options={fill=red},
    red
    ] table [x index=2, y index=1, col sep=comma]
   {big_plot_data.csv};
\end{axis}
\end{tikzpicture}